\documentclass{amsart}
\usepackage{amsmath,amsthm}
\usepackage{amsfonts,amssymb}

\usepackage{graphicx}
\usepackage{ifpdf}

\newtheorem{thm}{Theorem}[section]
\newtheorem{cor}[thm]{Corollary}
\newtheorem{lem}[thm]{Lemma}
\newtheorem{prop}[thm]{Proposition}

\theoremstyle{remark}
\newtheorem{rem}{Remark}[section]

 \def\a{{\alpha}} 
 \def\b{{\beta}}
 \def\g{{\gamma}}
 
 \def\t{{\theta}}
 
 \def\d{{\delta}}

 \def\s{{\sigma}}

 \def\la{{\langle}}
 \def\ra{{\rangle}} 
 \def\ve{{\varepsilon}}

 \def\jb{{\mathbf j}}
 \def\kb{{\mathbf k}}
 \def\lb{{\mathbf l}}
 
 \def\sb{{\mathbf s}}
 \def\tb{{\mathbf t}}
 \def\ub{{\mathbf u}}
 
 \def\xb{{\mathbf x}}

 \def\CA{{\mathcal A}}

 \def\CE{{\mathcal E}}
 
 \def\CH{{\mathcal H}}

 \def\CP{{\mathcal P}}     
      
 \def\CT{{\mathcal T}}

  \def\HH{{\mathbb H}}
  \def\JJ{{\mathbb J}}
 \def\NN{{\mathbb N}}

 \def\RR{{\mathbb R}}
 
 \def\ZZ{{\mathbb Z}}
 
 \newcommand{\e}{\mathrm{e}}
 \newcommand{\tr}{{\mathsf {tr}}}
 \newcommand{\TC}{{\mathsf {TC}}}
\newcommand{\TS}{{\mathsf {TS}}}
 \def\CTC{{\mathcal T\!C}}

	\def\sspan{\operatorname{span}}

\newcommand{\wt}{\widetilde}
\newcommand{\wh}{\widehat}

 \ifpdf
 \DeclareGraphicsExtensions{.pdf,.png,.jpg}
 \else
 \DeclareGraphicsExtensions{.eps}
 \fi
 \graphicspath{{figures/images/}}

\begin{document}

\title[Fourier series and approximation on hexagon and triangle]
{Fourier series and approximation on hexagonal and  triangular domains}

\author{Yuan Xu}
\address{Department of Mathematics\\ University of Oregon\\
    Eugene, Oregon 97403-1222.}\email{yuan@math.uoregon.edu}

\date{\today}
\thanks{The work was supported in part by NSF Grant DMS-0604056}

\begin{abstract}
Several problems on Fourier series and trigonometric approximation on a
hexagon and a triangle are studied. The results include Abel and Ces\`aro 
summability of Fourier series,  degree of approximation and best 
approximation by trigonometric functions, both direct and inverse theorems. 
One of the objective of this study is to demonstrate that Fourier series on 
spectral sets enjoy a rich structure that allow an extensive theory for Fourier 
expansions and approximation. 
\end{abstract}

\maketitle

\section{Introduction}
\setcounter{equation}{0}

A theorem of Fuglede \cite{F} states that a set tiles $\RR^n$ by lattice translation 
if and only if it has an orthonormal basis of exponentials $e^{i \la \a, x\ra}$ with 
$\a$ in the dual lattice. Such a  set is called a spectral set.  The theorem suggests 
that one can study Fourier series and approximation on a spectral set.  For the 
simplest spectral sets, cubes in $\RR^d$, we are in the familiar territory of classical 
(multiple) Fourier series. For other spectral sets, such a study has mostly remained 
at the structural property in the $L^2$ level and it seems to have attracted little 
attention among researchers in approximation theory. 

Besides the usual rectangular domain, the simplest spectral set is a regular 
hexagon on the plane, which has been studied in connection with Fourier
analysis in \cite{AB, Sun}. Recently in \cite{LSX}, discrete Fourier analysis on 
lattices was developed and the case of hexagon lattice was studied in detail; in
particular, Lagrange  interpolation and cubature formulas by trigonometric 
functions on a regular hexagon and on an equilateral triangle were studied. Here
we follow the set up in \cite{LSX} to study the summability of Fourier series and 
approximation. The purpose of this paper is to show, using the hexagonal domain
as an example, that Fourier series on a spectral set has a rich structure that 
permits an extensive theory of Fourier expansions and approximation.  It is 
our hope that this work may stimulate further studies in this area.  

It should be mentioned that, in response to a problem on construction and analysis
of hexagonal optical elements, orthogonal polynomials on the hexagon were 
studied in \cite{Dunkl} in which a method of generating an orthogonal polynomial 
basis was developed. We study orthogonal expansion and approximation by 
trigonometric functions on the hexagon domain. 

In comparison to the usual Fourier series for periodic functions in both variables
on the plane,  the periodicity of the Fourier series on a hexagonal domain is 
defined in terms of the hexagon lattice, which has the symmetry of the reflection
group $\CA_2$ with reflections along the edges of the hexagon. The functions 
that we consider are periodic under the translation of hexagonal lattice. It turns
out (\cite{LSX, Sun}) that it is convenient to use homogenous coordinates that 
satisfy $t_1 + t_2 +t_3 =0$ in $\RR^3$ rather than coordinates in $\RR^2$. 
Using homogenous coordinates allows us to treat the three directions equally
and it reveals symmetry in various formulas that are not obvious in  $\RR^2$ 
coordinates. As we shall show below, many results and formulas resemble closely
to those of Fourier analysis and approximation on $2\pi$ periodic functions. 

Fourier analysis on the hexagonal domain can be approached from several 
directions. Orthogonal exponentials on the hexagonal domain are related to 
trigonometric functions on an equilateral triangle, upon considering symmetric 
exponentials on the hexagonal domain. These trigonometric functions arise 
from solutions of Laplacian on the equilateral triangle, as seen in \cite{K} and 
developed extensively in \cite{Sun, LS}, and they are closely related to 
orthogonal algebraic polynomials on the domain bounded by Steiner's 
hyercycloid \cite{K, LSX}, much as Chebyshev polynomials arise from 
exponentials. In fact, the trigonometric functions arise from the exponentials 
by symmetry are called generalized cosine functions in \cite{LSX,Sun}, and 
there are also generalized sine functions that are anti-symmetric. Our results 
on the hexagonal domain can be easily translated to results in terms of 
generalized cosines. Our results on summability can also be translated to 
orthogonal expansions of algebraic polynomials on the domain bounded by 
hypercycloid, but the same cannot be said on our results on best approximation. 
In fact, just like the case of best approximation by polynomial on the interval, 
the approximation should be better at the boundary for polynomial 
approximation on the hypercycloid domain. For example, our Bernstein type 
inequality (Theorem 4.8) can be translated into a Markov type inequality 
for algebraic polynomials. A modulus of smoothness will need to be defined
to take into account of the boundary effect of the hypercycloid domain, 
which is not trivial and will not be considered in this paper. Some of our 
results, especially those on the best approximation, can be extended to higher 
dimensions. We choose to stay on the hexagonal domain to keep an
uniformity of the paper and to stay away from overwhelming notations. 

The paper is organized as follows. Definitions and background materials will
be given in Section 2.  In Section 3 we study the Abel summability, aka Poisson 
integral, and Ces\`aro $(C,\delta)$ means of the Fourer series on the hexagon,
where several compact formulas for the kernel functions will be deduced. One 
interesting  result shows that the $(C,2)$ means are nonnegative, akin to the
Fej\`er means for the classical Fourier series. In Section 4 we study best 
approximation by trigonometric functions on the hexagonal domain and 
establish both direct and inverse theorems in terms of a modulus of 
smoothness.

\section{Fourier series on the regular hexagon}
\setcounter{equation}{0}

Below we briefly sum up what we need on Fourier analysis on hexagonal 
domain. We refer to \cite{LSX} for further details. The hexagonal lattice 
is given by $H \ZZ^2$, where the matrix $H$ and the spectral set
$\Omega_H$ are given by
$$
H=\begin{pmatrix} \sqrt{3} & 0\\ -1 & 2\end{pmatrix}, \qquad 
\Omega_H =\left\{(x_1,x_2):\  -1\leq x_2, \tfrac{\sqrt{3}}{2}x_1 \pm 
   \tfrac{1}{2} x_2 < 1 \right\},
$$ 
respectively. The reason that $\Omega_H$ contains only half of its 
boundary is given in \cite{LSX}. We will use homogeneous coordinates
 $(t_1,t_2,t_3)$ that satisfies $t_1 + t_2 +t_3 =0$ for which the hexagonal
 domain $\Omega_H$ becomes 
\begin{align*}
\Omega=\left\{(t_1,t_2,t_3)\in \RR^3:\  -1\le  t_1,t_2,-t_3<1;\, 
    t_1+t_2+t_3=0 \right\},
\end{align*}
which is the intersection of the plane $t_1+t_2+t_3=0$ with the cube $[-1,1]^3$,
as seen in Figure 1.

\begin{figure}[h]
\centerline{\includegraphics[width=0.4\textwidth]{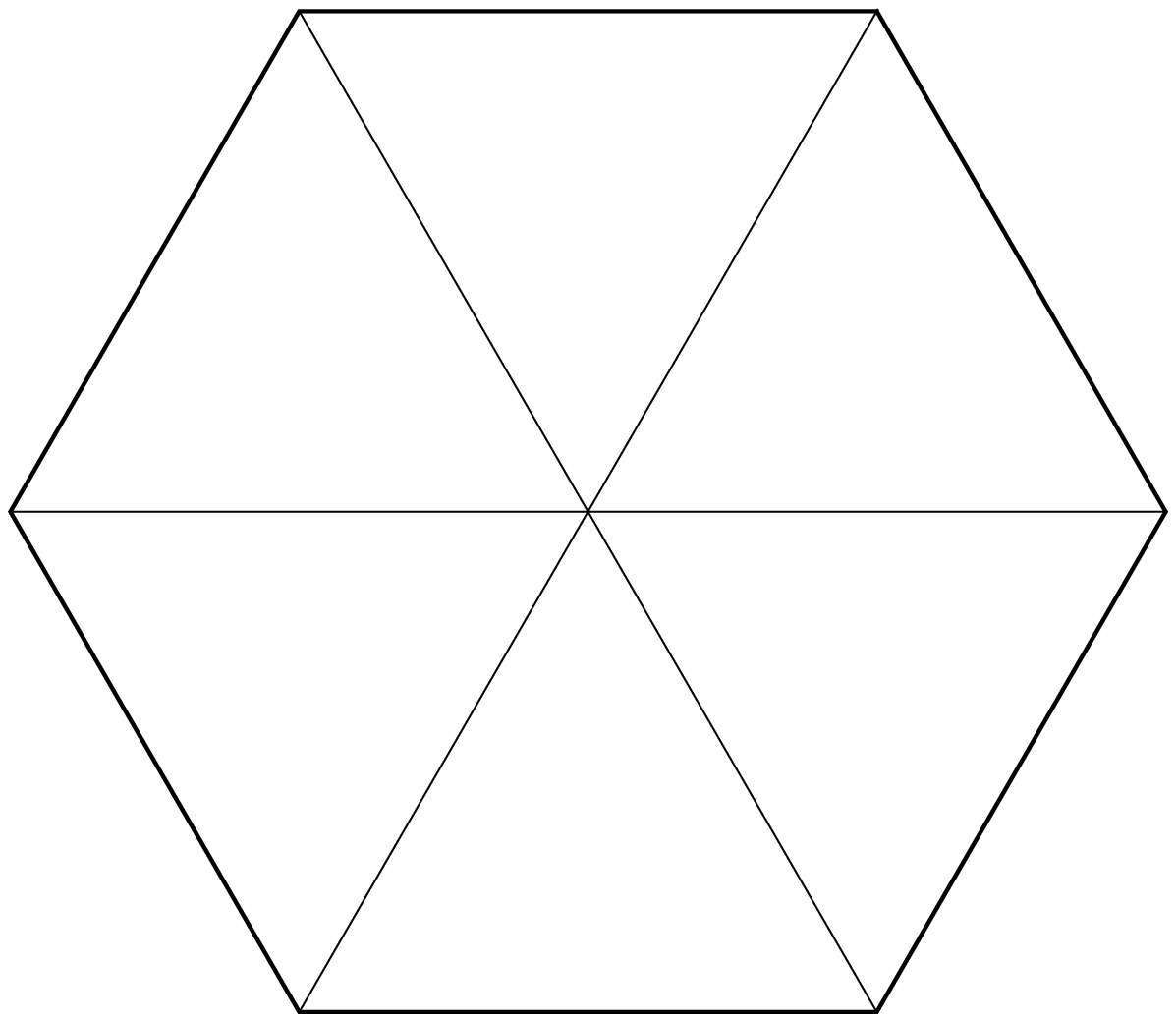} \qquad\quad
\includegraphics[width=0.375\textwidth]{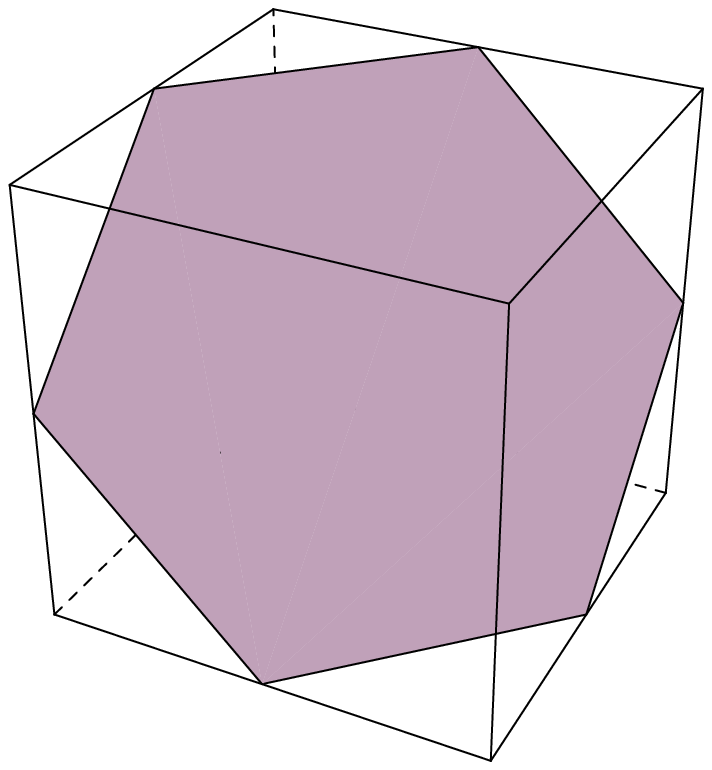}}
\caption{Regular hexagon in $\RR^2$ and in $\RR^3$.}
\end{figure}

The relation between $(x_1,x_2) \in \Omega_H$ and $\tb \in \Omega$ is
given by 
\begin{align}\label{coordinates}
 t_1= -\frac{x_2}{2} +\frac{\sqrt{3}x_1}{2},\quad t_2 = x_2, \quad t_3 = 
      -\frac{x_2}{2} -\frac{\sqrt{3}x_1}{2}. 
\end{align}
For convenience, we adopt the convention of using bold letters, such 
as $\tb$, to denote points in the space 
$$
    \RR_H^3 : = \{\tb = (t_1,t_2,t_3)\in \RR^3: t_1+t_2 +t_3 =0\}. 
$$
In other words, bold letters such as  $\tb$ stand for homogeneous coordinates.  
If  we treat $x \in \RR^2$ and $\tb \in \RR_H^3$ as column vectors,  then it 
follows from \eqref{coordinates} that 
\begin{equation} \label{x-t-x}
x = \tfrac13  H (t_1 - t_3, t_2-t_3)^\tr =  \tfrac13  H (2 t_1 + t_2, t_1+2t_2)^\tr 
\end{equation}
upon using the fact that $t_1 + t_2 + t_3 =0$. Computing the Jacobian of the 
change of variables shows that $d x = \frac{2 \sqrt{3}} {3} d t_1 dt_2$. 

A function $f$ is called {\it periodic} with respect to the hexagonal lattice, if 
$$ 
        f(x) = f (x + H k), \qquad k \in \ZZ^2. 
$$
We call such a function $H$-periodic. In homogeneous coordinates, 
$x\equiv y \pmod{H}$ becomes, as easily seen using \eqref{x-t-x}, 
$\tb  \equiv \sb \mod 3$, where we define
$$
 \tb \equiv \sb \mod 3 \quad \Longleftrightarrow \quad  t_1-s_1 \equiv t_2-s_2 
      \equiv t_3-s_3 \mod 3.
$$
Thus, a function $f (\tb)$ is $H$-periodic if $f (\tb) = f(\tb + \jb)$ whenever  
$\jb  \equiv 0 \mod 3$.  If $f$ is $H$-periodic, then it can be verified 
directly that 
\begin{equation} \label{IntPeriod}
 \int_{\Omega} f(\tb  + \sb) d\tb =   \int_{\Omega} f(\tb) d\tb, \qquad \sb \in \RR_H^3. 
\end{equation}

We define the inner product on the hexagonal domain by
\begin{align*}
\langle f, g\rangle_H := \frac{1}{|\Omega_H|} \int_{\Omega_H} f(x_1,x_2) 
   \overline{g(x_1,x_2)} d x_1 dx_2
    = \frac{1}{|\Omega|} \int_{\Omega} f(\tb) \overline{g(\tb)} d \tb,
\end{align*}
where $|\Omega|$ denote the area of $\Omega$. Furthermore, let 
$\ZZ_H^3: = \ZZ^3 \cap \RR_H^3$. We define
$$
\phi_\jb(\tb) : = \e^{\frac{2 \pi i}{3}\jb \cdot \tb}, \qquad \jb \in 
      \ZZ_H^3, \quad \tb \in \RR_H^3.
$$
These exponential functions are orthogonal with respect to $\la f,g\ra_H$.  

\begin{thm} \cite{F}. 
For $\kb,\jb \in \ZZ_H^3$, $\la\phi_\kb, \phi_\jb \ra = \delta_{\kb,\jb}$. 
Moreover, the set $\{\phi_\jb: \jb \in \ZZ_H^3\}$ is an orthonormal basis of
$L^2(\Omega)$.
\end{thm}

It is easy to see that $\phi_\jb$ are $H$-periodic functions. We  consider a 
special collection of them indexed by $\jb$ inside a hexagon. Define
$$
   \HH_n : = \{\jb \in \ZZ_H^3:  -n \le j_1,j_2,j_3 \le n\} \quad\hbox{and} \quad 
       \JJ_n: = \HH_n \setminus \HH_{n-1}.  
$$ 
Notice that $\jb \in \HH_n$ satisfies $j_1 + j_2 + j_3 =0$, so that $\HH_n$
contains all integer points inside the hexagon $n\, \overline{\Omega}$, whereas
$\JJ_n$ contains exactly those integer points in $\ZZ_H^3$ that are on the 
boundary of $n\, \Omega$, or $n$-th hexagonal line. We then define
\footnote{Here our notation differs from those in \cite{LSX}. Our $\HH_n$
and $\CH_n$ are in fact $\HH_n^*$ and $\CH_n^*$ there.}
\begin{equation} \label{Hn-space}
   \CH_n: = \sspan \left \{ \phi_\jb: \jb \in \HH_n \right \}. 
\end{equation}
It follows  that $\dim \CH_n = 3 n^2  +3n+1$. As we shall see below, the class 
$\CH_n$ shares many properties of the class of trigonometric polynomials 
of one variable. As a result, we shall call functions in $\CH_n$ trigonometric 
polynomials over $\Omega$. We will study the best approximation by 
trigonometric polynomials in $\CH_n$ in Section 4. 

By Theorem 2.1, the standard Hilbert space theory shows that an $H$-periodic 
function $f\in L^2(\Omega)$ can be expanded into a Fourier series
\begin{equation}\label{FourierSeries}
  f  = \sum_{\jb \in \ZZ_H^3} \wh f_\jb \phi_\jb  \quad \hbox{in $L^2(\Omega)$},
       \quad \hbox{where} \quad 
    \wh f_\jb :=  \frac{1}{|\Omega|} \int_{\Omega} f(\sb) e^{- \frac{2\pi i \sb}{3}} d\sb.
\end{equation}
We consider the $n$-th hexagonal Fourier partial sum defined by 
\begin{equation} \label{PartialSum} 
 S_n f(x) := \sum_{\jb \in \HH_n} \wh f_\jb  \phi_\jb(\tb) = 
       \frac{1}{|\Omega|} \int_{\Omega} f(\tb - \sb) D_n(\sb) d\sb,
\end{equation}
where the second equal sign follows from \eqref{IntPeriod} with the kernel
$D_n$ defined by 
$$
    D_n(\tb) : = \sum_{\jb \in \HH_n} \phi_\jb(\tb).  
$$
The kernel is an analogue of the Dirichlet kernel for the ordinary Fourier 
series. It enjoys a compact formula given by \cite{Sun} (see also \cite{LSX}) that 
\begin{equation}\label{D-kernel}
   D_n(\tb ) = \Theta_n(\tb) - \Theta_{n-1}(\tb), 
\end{equation}
where
\begin{equation}\label{Theta}
\Theta_n(\tb) = \frac{\sin \frac{(n+1) (t_1-t_2)\pi}{3}\sin \frac{(n+1) (t_2-t_3)\pi}{3}
  \sin \frac{(n+1) (t_3-t_1)\pi}{3}}
    {\sin \frac{(t_1-t_2)\pi}{3}\sin \frac{(t_2-t_3)\pi}{3}\sin \frac{(t_3-t_1)\pi}{3}}.
\end{equation}
This last formula is our starting point for studying summability of Fourier 
series in the following section. 

We will also need to use the Poisson summation formula associated with
the hexagonal lattice. This formula takes the form (see, for example, 
\cite{AB, H, LSX})
\begin{equation} \label{Poisson-pre}
  \sum_{k \in \ZZ^2} f(x + H k) = \frac{1}{\det (H)} \sum_{j \in \ZZ^2} 
     \wh f ( H^{-\tr} k) e^{2 \pi i k^\tr H^{-1} \xb}, 
\end{equation} 
where for $f\in L^1(\RR^2)$ the Fourier transform $\wh f$ and its inverse are 
defined by
$$
   \wh f(\xi) = \int_{\RR^2} f(x) \e^{-2\pi \xi \cdot x} dx \quad \hbox{and} \quad
   f(x) = \int_{\RR^2} \wh f(\xi) \e^{2\pi \xi \cdot x} dx,
$$
and the formula holds under the usual assumption on the convergence of the 
series in both sides.  Define Fourier transform in homogeneous coordinates 
by $\wh f(\tb) := \wh f(H^{-\tr} t)$ with $t = (t_1,t_2)$ as a column vector. Using 
\eqref{x-t-x}, it is easy to see that $\wh f$ and its inverse become, in 
homogeneous coordinates, 
\begin{equation} \label{FT}
  \wh f(\sb) =  \int_{\RR_H^3} f(\tb) \e^{- \frac{2\pi i}{3} \sb \cdot \tb} d\tb
   \qquad \hbox{and} \qquad
  f(\tb) =  \int_{\RR_H^3} \wh f(\sb) \e^{\frac{2\pi i}{3} \sb \cdot \tb} d\sb. 
\end{equation} 
Next we reformulate Poisson summation formula in homogenous coordinates. 
The left hand side of \eqref{Poisson-pre} is an $H$-periodic function over 
hexagonal lattice, which becomes the summation of $f(\tb + 3 \jb)$ over all 
$\jb \in \ZZ_H^3$ as $x \equiv y \mod H$ becomes $\tb \equiv \sb \mod 3$. 
For the right hand side, using the fact that $(k_1,k_2) H^{-1} x = 
\frac{1}{3} (k_1,k_2) (t_1-t_3, t_2-t_3)^\tr  = \frac{1}{3} \kb \cdot \tb$ 
by \eqref{x-t-x}, we obtain 
$$
  \wh f(H^{-\tr} k) =  \int_{\RR^2} f(x) \e^{-2 \pi k^\tr H^{-1} x} dx 
   = \frac{2 \sqrt{3}}{3} 
      \int_{\RR_H^3} f(\tb) \e^{-2 \pi \kb \cdot \tb} d\tb  =
            \frac{2 \sqrt{3}}{3} \wh f(\kb).
$$
Consequently, we conclude that the Poisson summation formula in 
\eqref{Poisson-pre} becomes, in homogeneous coordinates, 
\begin{equation} \label{Poisson}
 \sum_{\kb \in \ZZ_H^3} f(\tb + 3 \kb) = \frac{1}{3} \sum_{\jb \in \ZZ_H^3}
       \wh f(\jb) \e^{\frac{2\pi i}{3} \jb\cdot \tb}. 
\end{equation}

Throughout of this paper we will reserve the letter $c$ for a generic 
constant, whose value may change from line to line. By $A \sim B$ we
mean that there are two constants $c$ and $c'$ such that $c A \le B
 \le c' A$.

\section{Summability of Fourier series on Hexagon}
\setcounter{equation}{0}

We consider hexagonal summability of Fourier series \eqref{FourierSeries};  
that is, we write the Fourier series \eqref{FourierSeries} as blocks whose 
indices are grouped according to $\JJ_n$: 
$$
     f(\tb) = \sum_{n=0}^\infty \sum_{\jb \in \JJ_n} \wh f_\jb \phi_{\jb}(\tb).   
$$
From now on we call such a series hexagonal Fourier series. Its $n$-th
partial sum is exactly $S_n f$ given in \eqref{PartialSum}.
 
\subsection{Abel summability}
We consider Poisson integral $P_r f(t)$ of the hexagon Fourier series, which
is defined by
$$
   P_r f( \tb):=\sum_{n =0}^\infty \sum_{\jb \in \JJ_n} \wh f_\jb \phi_\jb(\tb) r^n
      =   \frac{1}{\Omega}  \int_{|\Omega|} P(r; \tb - \sb) f(\sb) d\sb,
$$
where $P(r; \tb)$ denotes the Poisson kernel 
$$
   P(r; \tb) :=  \sum_{n=0}^\infty \sum_{\jb \in \JJ_n} \phi_\jb(\tb) r^n, 
           \qquad  0 \le r < 1.
$$
Just as in the classical Fourier series, the kernel is nonnegative and enjoys 
a closed expression. 

\begin{prop}
(1) The Poisson kernel $P(r; \tb)$ is nonnegative for all $\tb \in \Omega$ and
$$
       \frac{1}{|\Omega}  \int_{\Omega} P(r; \tb) d \tb =1.
$$
(2) Let $q(r,t) = 1 - 2r \cos t + r^2$. Then 
\begin{align*}
  P(r; \tb) = & \frac{ (1-r)^3(1-r^3) }
     { q\left(r,\frac{2 \pi(t_1-t_2)}{3} \right) q\left(r,\frac{2 \pi (t_2-t_3)}{3}\right) 
          q\left(r,\frac{2 \pi (t_3-t_1)}{3}\right)  } \\
            &  +  \frac{r (1-r)^2 } { q\left(r,\frac{2 \pi (t_1-t_2)}{3}\right) 
                  q\left(r,\frac{2 \pi (t_2-t_3)}{3}\right)   }  
            +  \frac{ r(1-r)^2 } {   q\left(r,\frac{2 \pi(t_2-t_3)}{3} \right) 
                 q\left(r,\frac{2 \pi (t_3-t_1)}{3}\right) }  \\
            &  +  \frac{r (1-r)^2}  {  q\left(r,\frac{2 \pi(t_3-t_1)}{3} \right)
               q\left(r,\frac{2 \pi (t_1-t_2)}{3}\right)   } .
\end{align*}
\end{prop} 

\begin{proof}
The fact that the integral of $P(r;\tb)$ is 1 follows from the definition and 
the orthogonality of $\phi_\jb(t)$, whereas that $P(r;\tb) \ge 0$ is an immediate
consequence of the compact formula in the part (2). 

To prove the part (2), we start from the compact formula of $D_n(\tb)$, from which
follows readily that 
$$
  P(r;\tb) = (1-r) \sum_{n=0}^\infty D_n(\tb) r^n = (1-r)^2 \sum_{n=0}^\infty
          \Theta_n(\tb) r^n. 
$$
If $t_1+t_2+t_3 =0$ then it is easy to verify that 
\begin{align} \label{elementary} 
\begin{split}
    & \sin 2 t_1+  \sin 2 t_2+  \sin 2 t_3 =  - 4 \sin t_1 \sin t_2 \sin t_3,\\
    & \cos 2 t_1+  \cos 2 t_2+  \cos 2 t_3 =   4 \cos t_1 \cos t_2 \cos t_3 -1.
\end{split}
\end{align}
Using the first equation in \eqref{elementary} and the fact that 
$$
   \sum_{n=0}^\infty \sin (n+1) s = \frac{\sin s} {1 - 2r \cos s + r^2},
$$
we conclude then 
\begin{align*}
\sum_{n=0}^\infty \Theta_n(\tb) r^n =  & \frac{1}
    {4 \sin \frac{\pi(t_1-t_2)}{3}\sin  \frac{\pi(t_2-t_3)}{3}\sin \frac{\pi (t_3-t_1)}{3} } \\
       & \times \left[ \frac{\sin \frac{2\pi (t_1-t_2)}{3}} 
          { q\left(r,\frac{2 \pi(t_1-t_2)}{3} \right) } 
           +\frac{\sin\frac{2\pi (t_2-t_3)}{3} }{q\left(r,\frac{2 \pi (t_2-t_3)}{3}\right) }+
                  \frac{\sin\frac{2\pi(t_3-t_1)}{3}} {q\left(r,\frac{2 \pi (t_3-t_1)}{3}\right) } \right].
\end{align*}
Putting the three terms together and simplify the numerator, we conclude,
after a tedious computation,  that 
\begin{align*}
 \sum_{n=0}^\infty \Theta_n(\tb) r^n =   \frac{1 + 2 r + 2 r^3 + r^4 - 2 r^2
     \left[ \cos \frac{2 \pi(t_1 - t_2)}{3} + \cos \frac{2 \pi( t_1 - t_3)}{3} + 
          \cos \frac{2 \pi (t_2 - t_3) }{3}  \right]}
     { q\left(r,\frac{2 \pi(t_1-t_2)}{3} \right) q\left(r,\frac{2 \pi (t_2-t_3)}{3}\right) 
          q\left(r,\frac{2 \pi (t_3-t_1)}{3}\right)  }.
\end{align*}
The numerator of the right hand side can be written as 
\begin{align*}
   (1-r) (1-r^3)+ r\left[ \left(r,\tfrac{2 \pi(t_1-t_2)}{3} \right) +
       q\left(r,\tfrac{2 \pi (t_2-t_3)}{3}\right) 
        + q\left(r,\tfrac{2 \pi (t_3-t_1)}{3}\right) \right],
\end{align*}
from which the stated compact formula follows. 
\end{proof}

Next we consider the convergence of $P_r f$ as $r\to 1-$. For the 
classical Fourier series, if $P_r f$ converges to $f$ as $r\to 1-$ then the
series is called Abel summable. 

\begin{thm}
If $f$ is an $H$-periodic function, bounded on $\overline{\Omega}$ and 
continuous at $\tb \in \Omega^\circ$, then $P_r f(\tb)$ converges to $f(\tb)$ 
as $r \mapsto 1-$. 
Furthermore, if $f$ is continuous on $\overline{\Omega}$, then $P_r f$ 
converges uniformly on $\Omega$ as $r \mapsto 1-$. 
\end{thm} 

\begin{proof} 
Since $f$ is continuous at $\tb \in \Omega^\circ$. For any $\ve > 0$, choose 
$\d > 0$ such that 
$$
    |f (\tb- \sb) - f(\tb)| < \ve \qquad \hbox{whenever $\sb \in 
          \Omega_\d: = \{\sb \in \Omega: |s_i| \le \d, \, 1 \le i \le 3\}$}. 
$$
Since $P(r; \sb)$ has a unit integral, it follows that 
\begin{align*}
|P_r f(\tb) - f(\tb) |& \le \ve \int_{\Omega_\d}P(r;\sb) d\sb 
 + 2 \int_{\Omega \setminus \Omega_\d}   |f (\tb- \sb) - f(\tb)|P(r;\sb) d\sb\\
   & \le \ve  + 2 \|f\|_\infty \int_{\Omega \setminus \Omega_\d}P(r;\sb) d\sb.
\end{align*}
Thus, it suffices to show that the last integral goes to 0 as $r \mapsto 1+$. 
Since $q(r,t) = (1-r)^2 + 2 r \sin^2 \frac{t}{2} \ge (1-r)^2$, the closed formula
of the Poisson kernel shows that $P_r(\tb)$ is bounded by 
\begin{align*}
  P(r; \tb) &  \le   \frac{2 (1-r)^2 } { q\left(r,\frac{2 \pi (t_1-t_2)}{3}\right) 
                  q\left(r,\frac{2 \pi (t_2-t_3)}{3}\right)   }  
            +  \frac{ 2(1-r)^2 } {   q\left(r,\frac{2 \pi(t_2-t_3)}{3} \right) 
                 q\left(r,\frac{2 \pi (t_3-t_1)}{3}\right) }  \\
            &  +  \frac{2 (1-r)^2}  {  q\left(r,\frac{2 \pi(t_3-t_1)}{3} \right)
               q\left(r,\frac{2 \pi (t_1-t_2)}{3}\right)   } .
\end{align*}
Clearly we only need to consider one of the three terms in the right hand side,
say the second one, so that the essential task is reduced to show that, as
$q(r,t)$ is even in $t$,  
\begin{equation} \label{limit1}
   \lim_{r \to 1+} \int_{\Omega \setminus \Omega_\d} \frac{ (1-r)^2 } 
   { q\left(r,\frac{2 \pi (t_1-t_3)}{3}\right) q\left(r,\frac{2 \pi (t_2-t_3)}{3}\right)} d\sb
       =0. 
\end{equation}
Setting $P(r; t): = \frac{1-r^2}{q(r,t)}$ for $- \pi \le t \le \pi$. Then $P(r;t)$ is the 
Poisson kernel for the classical Fourier series. It is known \cite{Z} that 
\begin{equation} \label{limit2}
     \int_{-\pi}^\pi P(r;t) dt =1 \qquad\hbox{and}\qquad
        0 \le P(r;t) \le  c \frac{1-r} {t^2},
\end{equation} 
where $c$ is a constant independent of $r$ and $t$. Evidently, the kernel in 
the left hand side of \eqref{limit1} can be written as 
$$
    \frac{ (1-r)^2 } 
      { q\left(r,\frac{2 \pi (t_1-t_3)}{3}\right) q\left(r,\frac{2 \pi (t_2-t_3)}{3}\right)}
        =\frac{1}{(1+r)^2} P\left(r;\tfrac{2 \pi (t_1-t_3)}{3}\right) 
                     P\left(r,\tfrac{2 \pi (t_2-t_3)}{3}\right).
$$
Recall that $t_1 + t_2 + t_3 =0$ for $\tb = (t_1,t_2,t_3) \in \Omega$. It is easy
to see that $\tb \in \overline{\Omega}$ means that $-1 \le t_1, t_2, t_1-t_2 \le 1$.
Let 
\begin{equation} \label{t-s}
s_1 := \frac{t_1 - t_3}{3} = \frac{2t_1 + t_2} {3} , \qquad s_2 := \frac{t_2 - t_3}{3}
  = \frac{t_1 +2 t_2} {3}.
\end{equation}
A simple geometric consideration shows that the image of the domain 
$\Omega_\d$ under the affine mapping $(t_1,t_2) \mapsto (s_1,s_2)$ 
in \eqref{t-s} contains the square $[-\d/6,\d/6]^2$.  Consequently, the image 
of $\overline{\Omega} \setminus \Omega_\d$ is a subset of 
$[-1,1]^2 \setminus [-\d/6,\d/6]^2$. Hence, changing variables $(t_1,t_2)
 \mapsto (s_1,s_2)$, we obtain by \eqref{limit2} that  
\begin{align*}
 & \int_{\Omega \setminus \Omega_\d} \frac{ (1-r)^2 } 
    {q\left(r,\frac{2 \pi (t_1-t_3)}{3}\right) q\left(r,\frac{2 \pi (t_2-t_3)}{3}\right)} d\tb \\
 & \qquad\qquad
   \le    \frac{3}{(1+r)^2} \int_{[-1,1]^2 \setminus [-\d/6,\d/6]^2}
      {P\left(r,2 \pi s_1\right) P\left(r,2 \pi s_2\right)} d\sb \\
 & \qquad\qquad \le 3 \int_{1 \ge |s| \ge \delta/6} P(r, 2 \pi s)ds \le c \frac{1-r}{\d},   
\end{align*}
which converges to zero when $r \to 1-$.  This proves \eqref{limit1} and the 
convergence of $P_r f(\tb)$ to $f(\tb)$. Clearly it also proves the uniform 
convergence of $P_rf$ to $f \in C(\overline{\Omega})$.  
\end{proof}

\subsection{Ces\`aro summability}
Let us denote by $S_n^\d f$ and $K_n^\d$ the Ces\`aro $(C,\d)$ means of the 
hexagonal Fourier series  and its kernel, respectively.  Then
$$
   S_n^\d f(\tb) :=  \int_{\Omega} f(\sb) K_n^{(\d)}(tb - \sb) d\sb, 
$$
where 
$$
 K_n^{(\d)}(\tb):= \frac{1}{A_n^\d} 
     \sum_{k=0}^n A_{n-k}^\d \sum_{\jb \in \JJ_n} \phi_\jb(\tb), \qquad
        A_n^\d = \binom{n+\d}{\d}.
$$
It is evident that the case $\d =0$ corresponds to $S_n f$ and $D_n(\tb)$,
respectively. By \eqref{D-kernel} and \eqref{Theta}, the $(C,1)$ kernel
is given by
$$
   K_n^{(1)}(\tb) = \frac{1}{n} \Theta_n(\tb) = \frac{1}{n} 
       \frac{\sin \frac{(n+1) (t_1-t_2)\pi}{3}\sin \frac{(n+1) (t_2-t_3)\pi}{3}
  \sin \frac{(n+1) (t_3-t_1)\pi}{3}}
    {\sin \frac{(t_1-t_2)\pi}{3}\sin \frac{(t_2-t_3)\pi}{3}\sin \frac{(t_3-t_1)\pi}{3}}.
$$
For the classical Fourier series in one variable, the $(C,1)$ kernel is
the well-known Fej\`er kernel, which is given by
$$
       \frac{1}{n+1}  \sum_{k=0}^n D_k(t) = \frac{1}{n+1} 
           \left(\frac{\sin \frac{(n+1)t}{2}}{\sin \frac{t}{2} } \right)^2
$$
and is nonnegative in particular. For the hexagonal Fourier series, it turns 
out that $K_n^{(2)}$ is nonnegative. 

\begin{lem} For $n \ge 0$, 
$$
    \binom{n+2}{2} K_n^{(2)}(\tb) = \frac{1}{16}
    \frac{A_n(\tb)^2 + B_n(\tb)^2} {\left(\sin \frac{(t_1-t_2)\pi}{3}\right)^2
     \left( \sin  \frac{(t_2-t_3)\pi}{3}\right)^2\left(\sin \frac{(t_3-t_1)\pi}{3}\right)^2},       
$$
where
\begin{align*}
 A_n(\tb) &\, : =   \cos\tfrac{n t_1 \pi}{3} \sin \tfrac{(n+2) (t_2 - t_3)\pi}{3} + 
      \cos\tfrac{n t_2 \pi}{3} \sin \tfrac{(n+2) (t_3-t_1)\pi}{3} 
      + \cos\tfrac{n t_3\pi}{3} \sin\tfrac{(n+2) (t_1 - t_2)\pi}{3}, \\
 B_n(\tb) &\, : = \sin\tfrac{n t_1 \pi}{3} \sin \tfrac{(n+2) (t_2 - t_3)\pi}{3} + 
      \sin\tfrac{n t_2 \pi}{3} \sin \tfrac{(n+2) (t_3-t_1)\pi}{3}+ \sin\tfrac{n t_3\pi}{3} \sin\tfrac{(n+2) (t_1 - t_2)\pi}{3}.       
\end{align*}
\end{lem} 

\begin{proof}
Using \eqref{elementary} and the elementary formula 
$$
   \sum_{k=0}^n \sin [2 (k+1) s] = \frac{\sin [(n+1)s] \sin[(n+2)s]}{\sin s},
$$
we obtain
$$
 \sum_{k=0}^n \Theta_k(\tb) =  
  \frac{- 4 E_n(\tb)}
   {\left(\sin \frac{(t_1-t_2)\pi}{3}\right)^2
     \left( \sin  \frac{(t_2-t_3)\pi}{3}\right)^2\left(\sin \frac{(t_3-t_1)\pi}{3}\right)^2} .
$$
where
\begin{align*}
E_n(\tb) := & \sin \tfrac{(n+1) (t_1-t_2)\pi}{3} \sin \tfrac{(n+2) (t_1-t_2)\pi}{3}
     \sin  \tfrac{(t_2-t_3)\pi}{3} \sin \tfrac{(t_3-t_1)\pi}{3} \\
    + & \sin \tfrac{(n+1) (t_2-t_3)\pi}{3}\sin \tfrac{(n+2) (t_2-t_3)\pi}{3}
    \sin  \tfrac{(t_1-t_2)\pi}{3} \sin \tfrac{(t_3-t_1)\pi}{3} \\
    + & \sin \tfrac{ (n+1) (t_3-t_1)\pi}{3}\sin \tfrac{ (n+2) (t_3-t_1)\pi}{3} 
      \sin  \tfrac{(t_1-t_2)\pi}{3} \sin \tfrac{(t_2-t_3)\pi}{3}.
\end{align*}
From here, the difficulty lies in identifying that the numerator is a sum of 
squares. Once the form is recognized, the verification that $- 4 E_n(\tb) 
 = \frac{1}{16}(( A_n(\tb)^2 + B_n(\tb)^2)$ is a straightforward, though tedious, exercise.  
\end{proof}

An immediate consequence of the above lemma is the following: 

\begin{thm}
The $(C,2)$ means of the Fourier expansion with respect to the hexagon 
domain is a positive linear opreator. 
\end{thm}

As a comparison, let us mention that for the usual Fourier series on the tours, if
the partial sums are defined with respect to $\ell^\infty$ ball, that is, if the 
Dirichlet kernel is defined as 
$$  
  D_n(\t_1,\t_2)= \sum_{- n \le k_1,k_2\le n} e^{i(k_1 \t_1 + k_2 \t_2)}, \quad 
     (\t_1,\t_2) \in [-\pi,\pi]^2,
$$ 
then it is proved in \cite{BX} that the corresponding $(C,3)$ means are nonnegative 
and $\delta =3$ is sharp. In fact, the results in \cite{BX} are established for the 
partial sums defined with respect to the $\ell^1$ ball for $d$-dimensional torus. For 
$d=2$, it is easy to see that the result on $\ell^1$ ball implies the result on $\ell^\infty$ 
ball. 

If the $(C,\d)$ means converge to $f$, then so is $(C,\d')$ means for 
$\d' > \d$. The positivity of the kernel  shows immediately that $(C,2)$ means of the 
hexagonal Fourier series converges. It turns out that the $(C,1)$ means
is enough. 

\begin{thm}  \label{thm:(C,1)}
If $f \in C(\overline{\Omega})$, then the $(C,1)$ means $S_n^{(1)} f$ 
converge uniformly to $f$ in $\overline{\Omega}$. 
\end{thm}

\begin{proof} 
A standard argument shows that it suffices to prove that $S_n^{(1)}$ 
is a bounded operator, which amounts to show that 
$$
   I_n: = \int_\Omega \left| \Theta_n(\tb) \right| d\tb \le c\, n
$$
Since $\Theta_n(\tb)$ is a symmetric function in $t_1,t_2,t_3$, we
only need to consider the integral over the triangle
$$
\Delta: = \{\tb \in \RR_\HH:  0 \le t_1,t_2,-t_3 \le 1\} =
     \{(t_1,t_2): t_1 \ge 0, t_2 \ge 0, t_1+t_2 \le 1\},
$$
which is one of the six triangle in $\Omega$ (see Figure 1). Let 
$s_1,s_2$ be defined as in \eqref{t-s} and let $\Delta^*$ denote the 
image of $\Delta$ in $(s_1,s_2)$ plane. Then 
$$
\wt \Delta =\{(s_1,s_2): 0 \le s_1 \le 2 s_2, 0 \le s_2 \le 2 s_1, 
   s_1+s_2 \le 1\}
$$
and it follows, as the Jacobian of the change of variables is equal to 3, 
that 
$$
  I_n = 3 \int_{\wt \Delta} \left| \frac{\sin((n+1)\pi s_1) \sin( (n+1)\pi s_2)
       \sin ((n+1)\pi(s_1-s_2))}{\sin (\pi s_1) \sin (\pi  s_2)
         \sin(\pi  (s_1+s_2))} \right| ds_1ds_2.
$$
Since the integrant in the right hand side is clearly a symmetric function
of $(s_1,s_2)$, it is equal to twice of the integral over half of the 
$\wt \Delta$, say over
$$
\wt \Delta^* =\{(s_1,s_2) \in \wt \Delta: s_1 \le s_2\} 
=\{(s_1,s_2): s_1 \le s_2 \le 2 s_1,  s_1+s_2 \le 1\}. 
$$
Making another change of variables $s_1 = (u_1-u_2)/2$ 
and $s_2 = (u_1+u_2)/2$, the domain $\wt \Delta^*$ becomes
$\Gamma: = \{(u_1,u_2): 0 \le u_2 \le u_1/3, 0 \le u_1 \le 1\}$
and we conclude that
\begin{align*}
I_n  = 3 \int_{\Gamma} \left| \frac{\sin \frac{(n+1)(u_1+u_2)\pi}{2}
  \sin \frac{(n+1)(u_1-u_2)  \pi}{2} \sin\frac{(n+1)u_2 \pi}{2}}
    {\sin \frac{(u_1+u_2)\pi}{2} \sin\frac{(u_1-u_2)\pi}{2}
         \sin \frac{u_2\pi}{2}} \right| du_1 du_2 :=3 \int_{\Gamma} |\Theta_n^*(u)|du
\end{align*}
To estimate the last integral, we partition $\Gamma$ as $\Gamma = \Gamma_1 \cup \Gamma_2 \cup \Gamma_3$ and consider the three cases separately. 

\medskip\noindent
{\it Case 1.}  $\Gamma_1 =\{u \in \Gamma: u_1 \le 3/n\}$.  Using the fact
that $|\sin n t / \sin t| \le n$, we obtain
$$
   \int_{\Gamma_1}  |\Theta_n^*(u)|du 
         \le (n+1)^3 \int_{\Gamma_1} du_1 du_2 = (n+1)^3 \frac{3n^2}{2} \le c\, n.
$$

\medskip\noindent
{\it Case 2.} $\Gamma_2 =\{u \in \Gamma: 3/ n \le u_1, \, u_2\le 1/n\}$. 
In this region we have $u_1 - u_2 \ge 2 u_1 /3$ and $u_1 + u_2 \ge u_1$.
Hence, upon using $\sin u \ge (2/\pi) u$, we obtain
$$
   \int_{3/n}^1 \left| \frac{\sin \frac{(n+1)(u_1+u_2)\pi}{2}
  \sin \frac{(n+1)(u_1-u_2)  \pi}{2}  }
    {\sin \frac{(u_1+u_2)\pi}{2} \sin\frac{(u_1-u_2)\pi}{2}} \right|  du_1
      \le \int_{3/n}^1 \frac{1}{u_1^2} du_1 \le c\, n. 
$$
Consequently, it follows that 
$$
 \int_{\Gamma_2}  |\Theta_n^*(u)|du = \int_0^{1/n} \int_{3/n}^1|\Theta_n^*(u)|du 
    \le c \,n \int_{0}^{1/n} \left| \frac{\sin\frac{(n+1)u_2 \pi}{2}} {\sin \frac{u_2\pi}{2}} \right|
   \le c \, n
$$
upon using $|\sin n u / \sin u| \le n$ again. 

\medskip\noindent
{\it Case 3.} $\Gamma_3 =\{u \in \Gamma: 3/n \le u_, \, u_2 \ge 1/n\}$. 
In this region we have $u_1 \ge 3 u_2$, which implies that 
$u_1 - u_2 \ge (2/3) u_1$ and $u_1 + u_2 \ge u_1$. Thus, using 
$\sin u \ge (2/\pi) u$ again, we obtain
$$
   \int_{3 u_2}^1 \left| \frac{\sin \frac{(n+1)(u_1+u_2)\pi}{2}
       \sin \frac{(n+1)(u_1-u_2)  \pi}{2}  }
        {\sin \frac{(u_1+u_2)\pi}{2} \sin\frac{(u_1-u_2)\pi}{2}} \right|  du_1
  \le c \int_{3u_2}^1 \frac{1}{u_1^2} du_1 \le  \frac{c}{u_2}. 
$$
Consequently, using $\sin u_2 \ge (2/\pi)u_2$, we conclude that 
$$
 \int_{\Gamma_3}  |\Theta_n^*(u)|du = \int_0^{1/n} \int_{3/n}^1|\Theta_n^*(u)|du 
      \le c \int_{1/n}^{1/3} \frac{1}{u_2^2} d u_2   \le c \, n.
$$
Putting these estimates together completes the proof. 
\end{proof}

This theorem shows that $(C,\d)$ summability of the hexagonal Fourier series 
behaviors just like that of classical Fourier series. In particular, we naturally
conjecture that the $(C,\delta)$ means of the hexagonal Fourier series should 
converge if $\d > 0$. This condition is sharp as $\d =0$ corresponds to $S_nf$ 
whose norm is known to be in the order of $(\log n)^2$ (\cite{P, Sun}). 

\section{Best Approximation on Hexagon}
\setcounter{equation}{0}

For $1 \le p \le \infty$ we define $L^p$ space to be the space of Lebesgue 
integrable $H$-periodic functions on $\Omega$ with the norm 
$$ 
 \|f\|_p : = \left(\int_{\Omega} |f(\tb)|^p d \tb \right)^{1/p}, \qquad 1 \le p < \infty,
$$
and we assume that $L^p$ is $C(\overline{\Omega})$ when $p = \infty$, with 
the uniform norm on $\overline{\Omega}$. For $f \in L^p$, we define the error 
of best approximation to $f$ from $\CH_n$ by 
$$
   E_n (f)_p: = \inf_{ S\in \CH_n} \|f- S\|_p.
$$
We shall prove the direct and inverse theorems using a modulus of smoothness. 

\subsection{A simple fact}
We start with an observation that $E_n (f)$ can be related to the error of best
approximation by trigonometric polynomials of two variables on $[-1,1]^2$. To 
see this, let us denote by $\CT_n$ the space of trigonometric polynomials of 
two variables of degree $n$ in each variable, whose elements are of the form
$$
   T (u) = \sum_{k_1=0}^n \sum_{k_2=0}^n b_k \e^{2\pi i (k_1 u_1 + k_2 u_2)}. 
$$
Furthermore, if $f \in C([-1,1]^2)$ and it is $2\pi$ periodic in both variables, then
define
$$
 \CE_n(f) : = \inf_{T \in \CT_n}  \max_{u \in [-1,1]^2} |f(u) - T(u)|.
$$

\begin{prop}
Let $f$ be $H$-periodic and continuous over the regular hexagon 
$\overline{\Omega}_H$. Assume that $f^*$ is a continuous extension of $f$ 
on $[-1,1]^2$. Then 
$$
  E_n (f) \le  \CE_{\lfloor \frac{n}{2} \rfloor} (f^*). 
$$
\end{prop}

\begin{proof} 
We again work with homogeneous coordinates. Using the fact that 
$t_1+t_2+t_3=0$ and $j_1+j_2+j_3=0$, we can write
$$
   \phi_{\jb}(\tb) = e^{\frac{2\pi i}{3} \jb \cdot \tb} = 
       e^{2 \pi i (j_1 s_1 + j_2 s_2)}, \qquad s_1 = \tfrac{2 t_1 + t_2 }{3}, \quad
          s_2 = \tfrac{ t_1 + 2 t_2 }{3}. 
$$
Clearly $\tb \in \overline{\Omega}$ implies that $(s_1,s_2) \in \Omega^*: 
= \{s: -1 \le s_1, s_2, s_1+s_2 \le 1\}$. Furthermore,
$\jb \in \HH_n$ implies that $- n \le j_1,j_2, j_1+j_2 \le n$.  Consequently,
we have that 
$$
    S_n (\tb): = \sum_{j \in \CH_n} c_\jb \phi_\jb(\tb) = 
      \sum_{\substack {-n \le j_1,j_2 \le n \\
          -n \le j_1 + j_2 \le n} } c_{j_1,j_2,-j_1-j_2} 
            e^{2 \pi i  (j_1s_1 + j_2 s_2)} : = T_n(s). 
$$
Let $g(s) = f(\tb) =  f(2 s_1-s_2, 2s_2-s_1,-s_1-s_2)$ and let  its continuous 
extension to $[-1,1]^2$ be $g^*$. It follows that 
\begin{align*}
    \|f - S_n\| = \max_{s \in \Omega^*} | g(s) - T_n(s)|
          \le \max_{-1 \le s_1,s_2 \le 1}  | g^*(s) - T_n(s)|.
\end{align*}
Taking minimum over all $S_n \in \CH_n$ translates to taking minimal of all
$c_\jb$, consequently we conclude that 
\begin{align*}
   E_n (f) &\, \le  \min_{c_\jb} \max_{-1 \le s_1,s_2 \le 1}  \left | g^*(s) - 
        \sum_{\substack {-n \le j_1,j_2 \le n \\
          -n \le j_1 + j_2 \le n} }  c_\jb  \e^{2 \pi i  (j_1 s_1 + j_2 s_2)}    \right|  \\
               &\, \le  \min_{c_\jb} \max_{-1 \le s_1,s_2 \le 1}  \left | g^*(s) - 
        \sum_{-\lfloor \frac{n}{2}\rfloor \le j_1,j_2 \le \lfloor \frac{n}{2}\rfloor } 
               c_\jb  \e^{2 \pi i  (j_1 s_1 + j_2 s_2)}    \right|  
                   = \CE_{\lfloor \frac{n}{2}\rfloor} (g^*). 
\end{align*}
If we work with $f$ defined on $\Omega_H$, then $g^*$ becomes $f^*$ in the
statement of the proposition. 
\end{proof} 

For smooth functions, this result can be used to derive the convergence order 
of $E_n (f)$. The procedure of the proof, however, is clearly only one direction.
Below we prove both direct and inverse theorems using a modulus 
of smoothness. 

\subsection{Modulus of smoothness}
On the set $\Omega$, we can define a modulus of smoothness
by following precisely the definition for periodic functions on the
real line. Thus, we define for $r \in \NN_0$, 
$$
   \Delta_\tb f(\xb) = f(\xb + \tb) - f(\xb), \qquad \Delta_\tb^r f(\xb) = 
       \Delta_\tb \Delta_\tb^{r-1} f(\xb).
$$
It is well known that 
$$
 \Delta_\tb f (\xb) = \sum_{k=0}^r (-1)^k \binom{r}{k} f(\xb + k \tb).
$$
For $\tb \in \RR_\HH$, let $\|\tb\| := (t_1^2 +t_2^2)^{1/2}$, the usual
Euclidean norm of $(t_1,t_2)$. We can also take other norm of 
$(t_1,t_2)$ instead, for example, $\|\tb\|_\infty := \max\{|t_1|, |t_2|\}$.
Since $t_1+t_2+t_3 =0$, we have evidently 
$\|\tb\|_\infty \le \max\{|t_1|, |t_2|, |t_3|\} \le 2 \|\tb\|_\infty$. The modulus of smoothness of an $H$-periodic function $f$ is then 
defined as
$$
  \omega_r (f;h)_p := \sup_{\|\tb\|\le h} \|\Delta_\tb^r f \|_p, \qquad 1 \le p \le \infty. 
$$

\begin{prop} \label{prop:modulus}
The modulus of smoothness satisfies the following properties:
\begin{enumerate} 
\item For $\lambda > 0$, $\omega_r(f; \lambda h)_p \le  
     (1+\lambda)^r \omega_r(f;h)_p$.
\item Let $\partial_i$ denote the partial derivative with respect to $t_i$ and 
$\partial^k = \partial_1^{k_1}\partial_2^{k_2}\partial_3^{k_2}$ for 
$k = (k_1,k_2,k_3)$. Then 
$$
  \omega_r(f;h)_p \le h^r \sum_{k_1+k_2+ k_3 = r} \frac{r!}{k_1!k_2! k_3!}
       \|\partial^k f\|_p.  
$$
\end{enumerate}
\end{prop}

\begin{proof}
The proof of (1) follows exactly as in the proof of one variable. For part (2), it is 
easy to show by induction that 
$$
  \Delta_\tb^r f(\xb) = \int_{[0,1]^r} \partial_{u_1} \ldots \partial_{u_r} 
      f(\xb + u_1 \tb + \ldots + u_r \tb) du_1 \ldots u_r.
$$
The integrant of the right hand side is easily seen, by another induction, to be
$$
  \sum_{k_1+k_2+k_3 = r} \frac{r!}{k_1!k_2!k_3!} \partial^k 
        f(\xb + u_1 \tb + \ldots + u_r \tb) \tb^k, 
$$ 
from which the stated result follows from \eqref{IntPeriod} and the fact
that $\|\tb^k\| \le h^{|k|} = h^r$.  
\end{proof}

\subsection{Direct theorem}
For the proof of the direct theorem, we use an analogue of Jackson integral.
Let $r $ be a positive integer. We consider the kernel
$$
K_{n,r}(\tb):= \lambda_{n,r} \left[\Theta_n(\tb)\right]^{2r}, \qquad 
    \hbox{where} \qquad
  \int_\Omega K_{n,r}(\tb) d\tb =1. 
$$
Since $n^{-1} \Theta_n(\tb)$ is the $(C,1)$ kernel of the Fourier series, 
we see that $\Theta_n \in \CH_n$ and, thus, $K_{n,r} \in \CH_{r n}$. 

\begin{lem} \label{lem:main}
For $\nu \in \NN$ and $\nu \le 2 r -2$, 
$$
  \int_{\Omega} \|\tb\|^\nu K_{n,r} (\tb)  d \tb \le c n^{-\nu}.  
$$
\end{lem}

\begin{proof}
First we estimate the constant $\lambda_{n,r}$. We claim that 
\begin{equation} \label{lambda}
  \left( \lambda_{n,r} \right)^{-1} = \int_\Omega  \left[\Theta_n(\tb)\right]^{2r}
        d\tb \sim n^{-6 r +2}. 
\end{equation} 
We derive the lower bound first. As in the proof of Theorem \ref{thm:(C,1)}, 
we change variables from $\tb \in \Omega$ to $(s_1,s_2) \in \Omega^*$,
then use symmetry and change variables to $(u_1, u_2) \in \Gamma$. 
The result is that 
$$
 \int_\Omega  \left[\Theta_n(\tb)\right]^{2r} d\tb =
     3 \int_\Gamma \left[\Theta_n^*(\ub)\right]^{2r} d\ub \ge  
      3 \int_{\Gamma^*} \left[\Theta_n^*(\ub)\right]^{2r} d\ub, 
$$
where we choose $\Gamma^* =\{(u_1,u_2): \frac{1}{16 (n+1)} \le u_2 \le u_1/3 
 \le \frac{1}{8(n+1)} \}$, which is a subset of $\Gamma$ and its area is in the
order of $n^{-2}$.  For $u \in \Gamma^*$, 
we have $\sin (n+1) (u_1-u_2) \ge \sin \pi/16$,  $\sin (n+1) (u_1+u_2) 
 \ge \sin \pi/8$, and $\sin (n+1) u_2 \ge \sin \pi/32$; furthermore, 
 $\sin (u_1-u_2) \le \sin 5\pi/(16n) \le 5 \pi/(16n)$,  $\sin (u_1+u_2)
\le \sin \pi/(2n) \le \pi /(2n)$, and $\sin u_2 \le \sin \pi/(8n) \le \pi / (8 n)$. 
Consequently, we conclude that 
$$
 \int_\Omega  \left[\Theta_n(\tb)\right]^{2r} d\tb \ge 
    c  n^{6 r} \int_{\Gamma^*} d u = c n^{6 r -2}. 
$$
This proves the lower bound of \eqref{lambda}. The upper bound will follow 
as the special case $\nu=0$ of the estimate of the integral $I_n^{r,\nu}$ 
below.  

We now estimate the integral 
$$
     I_n^{r,\nu} : = \int_\Omega |\tb|^\nu \left[\Theta_n(\tb)\right]^{2r} d\tb. 
$$
Again we follow the proof of Theorem \eqref{thm:(C,1)}  and make a change 
of variables from $\tb \in \Omega$ to $(u_1, u_2) \in \Gamma$. The change 
of variables shows that 
$t_1 = \tfrac{1}{6} \left[2(u_1-u_2)-(u_1+u_2)\right]$ and $t_2 = 
\tfrac{1}{6} \left[2(u_1+u_2)-(u_1-u_2)\right]$, which implies that 
$$ 
\|\tb\|_\infty =  \max \{ |t_1|, |t_2| \} \le \tfrac12 \max \left \{|u_1-u_2|, 
      |u_1+u_2| \right \}.
$$
Consequently, we end up with 
$$
    I_n^{r,\nu} \le c  \int_\Gamma \max \left \{|u_1-u_2|, |u_1+u_2| \right \}^\nu
       |\Theta_n^*(u)|^{2 r}  du.   
$$
Since $\max\{|a|,|b|\} \le |a|+|b|$, we can replace the $\max\{ ... \}$ term 
in the integrant by the sum of the two terms. The fact that $\nu \le 2 r -2$ 
shows that we can cancel $|u_1-u_2|^\nu$,  or $|u_1+u_2|^\nu$, with the 
denominator of $\Theta_n^*$. After this cancellation, the integral can be
estimated by considering three cases as in the proof of Theorem \ref{thm:(C,1)}. 
In fact, the proof follows almost verbatim. For example, in the case 2, we
end up, using $u_1 - u_3 \ge 2 u_1/3 $ and $u_1+u_2 \ge u_1$, that  
\begin{align*}
&   \int_{3/n}^1 \left[ \frac{\sin \frac{(n+1)(u_1+u_2)\pi}{2}} 
         {\sin \frac{(u_1+u_2)\pi}{2}} \right]^{2r}
                      \left[ \frac{\sin \frac{(n+1)(u_1-u_2)  \pi}{2}}             
                          { \sin\frac{(u_1-u_2)\pi}{2}} \right]^{2r -\nu} du_1 \\
 & \qquad\qquad  \qquad\qquad \qquad\qquad
       \le \int_{3/n}^1 \frac{1}{u_1^{4r-\nu}} du_1 \le c\, n^{4r-\nu-1}. 
\end{align*}
Consequently, it follows that 
\begin{align*}
 \int_{\Gamma_2}  |u_1-u_2|^\nu |\Theta_n^*(u)|du & = 
    \int_0^{1/n} \int_{3/n}^1  |u_1-u_2|^\nu |\Theta_n^*(u)|du \\
     & \le c \,n^{4r-\nu -1}
   \int_{0}^{1/n} \left| \frac{\sin\frac{(n+1)u_2 \pi}{2}} {\sin \frac{u_2\pi}{2}} \right|^{2r}
      du   \le c \, n^{6 r - \nu -2}.
\end{align*}
upon using $|\sin n u / \sin u| \le n$. The other two cases can be handled 
similarly. As a result, we conclude that  $I_n^{r,p} \le c n^{6 r - p -2}$. 
The case $p = 0$ gives the lower bound estimate
of \eqref{lambda}. The desired estimate is over the quantity 
$\lambda_{n,r} I_n^{r,p}$ and follows from our estimates.
\end{proof}

Using the kernel $K_{n,r}$ we can now prove a Jackson estimate: 

\begin{thm}
For $1 \le p \le \infty$ and for each $r = 1,2,\ldots$, there is a constant
$c_r$ such that if $f\in L^p$ then 
$$
    E_n (f)_p \le c_r \omega_r (f, \tfrac{1}{n})_p, \qquad n = 1, 2, ... . 
$$
\end{thm}

\begin{proof}
As in the proof of classical Jackson estimate for trigonometric polynomials on
$[0, 2\pi]$, we consider the following operator
$$
F_n^{\rho,r} f (\xb):= \int_{\Omega} J_{n,\rho}(\tb) \sum_{k=1}^r (-1)^{k-1} 
    \binom{r}{k}    f(\xb + k \tb) d\tb, 
$$
where $J_{n,\rho}(\tb) = K_{n^*,\rho}(\tb)$ with $n^* = \lfloor \frac{n}{\rho} \rfloor +1$,
and $\rho \ge (r+2)/2$. Evidently, $J_{n,\rho}(-\tb) = J_{n,\rho}(\tb)$. 
Using the fact that $J_{n,\rho} \in \CH_n$, we see that $F_{n,\rho} f$ can
be written as a linear combination of   
\begin{equation} \label{periodJ}
\int_{\Omega}  f(\xb + k \tb)  \phi_{\jb} (\tb) d\tb, \qquad \jb \in \CH_n, 
    \quad k =1,\ldots, r.
\end{equation}
As $f$ is $H$-periodic, so is $f(\xb + k \tb)$ as a function of $\tb$. Let
 $F_m = \sum_{\jb \in \CH_m} a_\jb \phi_\jb$  denote the $(C,1)$ means of 
 the Fourier series of $f$ over $\Omega$. Then $F_m$ converges to $f$ 
 uniformly on $\Omega$. If $\jb \ne - k \lb$ for some 
$\lb \in \CH_n$ then, using the fact that $\phi_\lb (\xb + k \tb) =
 \phi_\lb(\xb) \phi_{k\lb}(\tb)$,  we obtain that
\begin{align*}
  \int_\Omega f(\xb + k \tb)  \phi_{\jb} (\tb) d\tb &\,= \lim_{n\to \infty}
       \int_\Omega F_m(\xb + k \tb)  \phi_{\jb} (\tb) d\tb \\
     & \, =   \lim_{n\to \infty} \sum_{\lb \in \CH_n}  a_\lb  \phi_\lb(\xb)
        \int_{\Omega} \phi_{k\lb}(\tb)\phi_{\jb} (\tb) d\tb 
     =  0.   
\end{align*}
If $\jb = - k \lb$, then making a change of variables $\xb + k \tb = \sb$ shows that
$$
  \int_{\Omega}  f(\xb + k \tb)  \phi_{\jb} (\tb) d\tb = 
     \int_{\Omega}  f(\xb + k \tb)  \phi_{\lb} (- k \tb) d\tb = 
          \int_{\Omega}  f(\sb)  \phi_{\lb} (\xb - \sb) d\sb
$$
which is a trigonometric polynomial in $\xb$ in $\CH_n$. Consequently, we 
conclude that $F_n^{\rho,r} f$ is indeed a trigonometric polynomial in $\CH_n$. 

Since $J_{n,\rho}(\tb) = J_{n,\rho}(-\tb)$, it follows from (1) of Proposition
\ref{prop:modulus} and Minkowski's inequality that 
\begin{align*}
 \|F_n^{\rho,r} f -f \|_p &\, \le \left \| \int_\Omega K_{n,\rho} (\tb) 
           \Delta_\tb^r f(\cdot)d \tb \right \|_p 
       \le \int_\Omega K_{n,\rho} (\tb)  \omega_r(f; \|\tb\|)_p d \tb \\
   &\,   \le \omega_r(f; \tfrac{1}{n})_p \int_\Omega K_{n,\rho} (\tb) (1+n \|\tb\|)^r
        d \tb   \le c\, \omega_r(f; \tfrac{1}{n})_p, 
\end{align*}
where the last step follows from Lemma \ref{lem:main}.
\end{proof} 

For $1 \le p \le \infty $ and $r=1,2,\ldots$, define $W_p^r$ as the space of
$H$-periodic functions whose $r$-th derivatives belong to $L^p$. 

\begin{cor} 
For $1 \le p \le \infty$ and $r = 1, 2, \ldots$, if $f \in W_p^{r}$ then 
$$
   E_n (f)_p \le c n^{-r} \sum_{|k| = p} \|\partial^k f\|_p, \qquad n = 1,2, \ldots.  
$$
\end{cor} 

\subsection{Inverse Theorem} 
As in the classical approximation theory, the main task for proving an inverse
theorem lies in the proof of a Bernstein inequality. For this, we introduce an
operator that is of interest in its own right. Let $\eta$ be a nonnegative $C^\infty$ 
function on $\RR$ such that
$$
   \eta(t) = 1, \quad \hbox{if $0\le t \le 1$}, \quad \hbox{and}\quad 
     \eta(t) = 0, \quad \hbox{if $t \le 0$ or $t\ge 2$}.  
$$
We then define an operator $\eta_n f$ on $\Omega$ by
$$
  \eta_n f(\xb) := \int_\Omega f(\xb-\tb) \eta_n(\tb) d\tb, 
         \qquad \hbox{where}\quad
  \eta_n(\tb):= \sum_{k=0}^{2n} \eta\left(\frac{k}{n} \right) D_k (\tb), 
$$
where $D_k$ is the Dirichlet kernel in \ref{D-kernel}. Evidently, $\eta_n f \in 
\CH_{2n}$ and $\eta_n f = f$ if $f\in \CH_n$. Such an operator has been 
used by many authors, starting from \cite{Kam}. It is applicable for orthogonal
expansion in many different settings;  see, for example,  \cite{X05}. A standard
procedure of summation by parts leads to the fact that $\|\eta_n f \|$ is 
bounded. Consequently, using  the fact that $\eta_n f = f$ for $f \in \CH_n$, 
we have the following result. 

\begin{prop}
For $1 \le p \le \infty$, if $f \in L^p$ then
$$
    \|\eta_n f - f \|_p \le c E_n (f)_p, \qquad n = 1, 2, \ldots .  
$$
\end{prop}

This shows that for all practical purpose, $\eta_n f$ is as good as the polynomial
of best approximation. For our purpose, however, the more important fact is the 
following near exponential estimate of the kernel function $\eta_n(\tb)$. 
For $\a \in \NN_0^d$, write $|\a|=\a_1+\ldots+\a_d$.

\begin{lem}
For each $k= 1, 2, \ldots$, there exists a constant $c_k$ that depends on $k$,
such that
$$
 \partial^\alpha \eta_n (\tb) \le c_k \frac{n^{|\alpha|+2}}{(1+ n \|\tb\|)^k}, 
               \qquad \tb \in \overline{\Omega}.
$$
\end{lem}

\begin{proof}
The main tool of the proof is the Poisson summation formula \eqref{Poisson}
as used in the case of trigonometric series on the real line in \cite{PX}.
Let us introduce a notation that $|\tb|_H = \max\{|t_1|, |t_2|, |t_3|\}$. Since
$\JJ_n = \{ \jb \in \HH_n:  \hbox{ $|j_1| = n$ or  $|j_2|=n$ or $|j_3| =n$} \}$,
we can write $\eta_n(\tb)$ as 
$$
 \eta_n (\tb) = \sum_{k=0}^{2n} \eta\left(\frac{k}{n} \right) \sum_{\jb \in \JJ}
       \phi_\jb(\tb) = 
      \sum_{\jb \in \HH_n} \eta\left( \frac{|\jb|_H}{n} \right) \phi_\jb(\tb). 
$$
In particular, for $\a \in \NN_0^3$, we have 
$$
  \partial^\a  \eta_n(\tb) =   \left(\tfrac{2 \pi}{3}\right)^{|\a|} \sum_{\jb \in \HH_n} 
         \eta\left( \frac{|\jb|_H}{n} \right) \jb^\a \phi_\jb(\tb).
$$
An immediate consequence of this expression is that 
\begin{equation} \label{eta-bound1}
  \left |\partial^\a \eta_n(\tb) \right| \le c \|\eta\|_\infty \sum_{\jb \in \HH_n} \|\jb^\a\| 
      \le c \|\eta\|_\infty n^{|\a|+2}.
\end{equation}

Define $\Phi_n$ such that $\wh \Phi_n (\tb) = \eta \left(\frac{|t|_H}{n}\right)
 \tb^\a$. Then 
$
    \Phi_n(\xi)  =  \int_{\RR^3_H} \wh \Phi_n(\tb) e^{\frac{2\pi i}{3} \xi \cdot \tb} d\tb.
$
The definition of $\Phi_n$ and Poisson summation formula shows that 
\begin{equation} \label{eta-bound2}
  \eta_n (\tb) = \sum_{\jb \in \HH_n} \wh \Phi_n(\jb) \phi_\jb(\tb)
     = 2 \sqrt{3} \sum_{\jb \in \ZZ_H^3} \Phi_n (\tb + 3 \jb). 
\end{equation}
In order to estimate the right hand side we first derive an upper bound for 
$\Phi_n$. Since $\eta$ is a $C^\infty$ function and $\|\tb\|_H$ is differentiable 
except when one of the variable is zero,  taking derivatives in $L^1$ norm, 
we  end up with
$$
  \left(\tfrac{2 \pi i}{3}\right)^{|\b|} \tb^\b \Phi_n(\tb) =  \int_{\RR^3_H} \partial^\b 
              \wh\Phi_n(\sb) e^{\frac{2\pi i}{3} \sb \cdot \tb} d\sb 
    =  \int_{\RR^3_H}  \partial^\b \left [ \eta \left(\tfrac{|\sb|_H}{n}\right) \sb^\a \right ]
        e^{\frac{2\pi i}{3} \sb \cdot \tb} d\sb. 
$$
Each derivative of $\eta\left(\frac{|\tb|_H}{n}\right)$ yields a $n^{-1}$. For 
$\beta \in \NN_0^3$ and $k:=|\b| > |\a|$, we have 
\begin{align*}
  \partial^\b \left [ \eta \left(\tfrac{|\tb|_H}{n}\right) \tb^\a \right ]
     & = \sum_{|\g| \le k} \binom{k}{\g} n^{-k+|\g|} \eta^{(k-|\g|)} 
      \left(\tfrac{|\tb|_H}{n}\right)  \partial^{\g} \tb^\a \\
    & =  \sum_{|\g| \le |\a|} \binom{|\b|}{\g} n^{-k+|\g|} \eta^{(k-|\g|)}
       \left(\tfrac{|\tb|_H}{n}\right) \frac{\a!}{\g!} \tb^{\a -\g}, 
\end{align*}
where we have used muti-index notations that for $k \in \NN$ and 
$\a \in \ZZ^d$, $\alpha! = \a_1! \ldots \a_d!$ and $\binom{k}{\a} = 
k!/ (\a! (k-|\a|)!)$.  Hence, as $\eta^{(j)}$ is supported on 
$[1,2]$ for $j \ge 1$, we deduce that 
\begin{align*}
  \left|  \left(\tfrac{2 \pi i}{3}\right)^{k} \tb^\b \Phi_n(\tb) \right | 
  &  \le  c \, n^{-|\b|+|\a|} \sum_{|\g|\le |\a|}\| \eta^{(k-|\g|)}\|_\infty 
         \int_{n \le |\tb|_H \le 2 n}d\tb \\
  &  \le c \, n^{-k+|\a| +2} \sum_{j= k-|\a|}^{k} \|\eta^{(j)}\|_\infty. 
\end{align*}
Together with \eqref{eta-bound1}, we conclude then
$$
  \left| \Phi_n(\tb) \right | \le c_k \frac{n^{|\a|+2}}{(1+ n \|\tb\|)^{k} },
      \qquad c_k = c \max_{k-|\a| \le j\le k} \|\eta^{(j)}\|_\infty. 
$$

As a result of the estimate, we conclude from \eqref{eta-bound2} that 
$$
  \left| \eta_n(\tb)\right| \le c \sum_{\jb \in \ZZ_H^3} |\Phi_n(\tb + 3 \jb)|
   \le c_k  \sum_{\jb \in \ZZ_H^3}\frac{n^{|\a|+2}}{(1+ n \|\tb+ 3 \jb\|)^{k} }. 
$$
Since $\|\tb\| = \max\{|t_1|, |t_2|\} \le 1$ for $\tb \in \Omega$, we have 
$\|\tb + 3 \jb\| \ge 3 \|\jb\| - 1 \ge 2\|\jb\|$ if $\jb \ne 0$, and thus,
$1+ n \|\tb + 3 \jb\| \ge (1+ n) \|\jb\|$. Consequently
$$
 \left| \eta_n f(\tb)\right| \le c_k \frac{n^{|\a|+2}}{(1+ n \|\tb\|)^{k} } +
 \frac{n^{|\a|+2}}{(1+n)^k}  \sum_{0 \ne \jb \in \ZZ_H^3}\frac{1}{\|\jb\|^{k} }
    \le  c_k \frac{n^{|\a|+2}}{(1+ n \|\tb\|)^{k} }
$$
and the proof is completed. 
\end{proof}

\begin{rem}
The proof of the above estimate relies essentially on Poisson summation
formula, which is known to hold for all lattice $A\ZZ^d$ (see, for example,
\cite{H, LSX}). Thus, there is a straightforward extension of the above 
result with an appropriate definition of partial sums. For relevant results 
on lattices, see \cite{CS}. 
\end{rem}

As an application of this estimate, we can now establish the Bernstein 
inequality for hexagonal trigonometric polynomials.

\begin{thm}
If $\a \in \NN_0^3$, then for $1 \le p \le \infty$ there is a constant $c_p$ 
such that 
$$
\left \| \partial^\a S_n \right \|_p \le c_p n^{|\a|} \|S_n\|_p, 
    \qquad \hbox{for all  $S_n \in \CH_n$}.
$$
\end{thm}

\begin{proof} 
Recall that $\eta_n f \in \CH_{2n}$ and $\eta_n f =f $ for $f \in \CH_n$. We
have then 
$$
S_n(\tb)  = (\eta_n S_n)(\tb) = \int_{\Omega} S_n(\sb) \eta_n( \tb - \sb) d \sb. 
$$
For $ p =1$ and $p = \infty$, we then apply the previous proposition 
with $k =4$ to obtain
\begin{align*}
  \left\|\partial^\a S_n \right\|_p & \le \| S_n\|_p  \int_{\Omega} 
           |\partial^\a \eta_n(\sb)| d \sb 
       \le c \|S_n\|_p \int_{\Omega} \frac{n^{|\a|+2}}{(1+ n|\tb|)^4} d\tb  \\
     &  \le cn^{|\a} \|S_n\|_p \int_{\RR^2} \frac{1}{(1+ |\tb|)^4} dt_1d t_2        
        \le cn^{|\a} \|S_n\|_p, 
\end{align*}
which establishes the stated inequality for $p =1$ and $p=\infty$. 
The case $1  < p < \infty$ follows from the case of $p =1$ and $p=\infty$ by
interpolation.  
\end{proof}

It is a a standard argument by now that the Bernstein inequality yields the
inverse theorem. 

\begin{thm} 
There exists a constant $c_r$ such that for each function $f\in C(\overline{\Omega})$
$$
    \omega_r(f; h)_p \le c_r h^r \sum_{0 \le n\le h^{-1}} (n+1)^{r-1} E_n (f)_p.
$$
\end{thm}

\section{Approximation on Triangle}
\setcounter{equation}{0}

The hexagon is invariant under the reflection group $\CA_2$, generated by 
the reflections in the edges of its three paris of parallel edges. In homogeneous
coordinates, the three reflections $\sigma_1$, $\sigma_2$ and $\sigma_3$
are given by 
$$
\tb  \sigma_1 :=  -(t_1,t_3,t_2),  \quad \tb \sigma_2 := -(t_2,t_1,t_3),
     \quad \tb\sigma_3:= -(t_3,t_2,t_1).
$$
Indeed, for example, the reflection in the direction $(\sqrt{3},1)$ becomes 
reflection in the direction of $\alpha = (1,-2,1)$ in $\RR_\HH^3$, which is
easy to see,  using $t_1+t_2+t_3 =0$, as given by $\tb- 2 \frac{\langle \alpha,\tb \rangle}{\langle \alpha,\alpha \rangle}  \alpha  = -(t_2,t_1,t_3) = t \sigma_2$, 
The reflection group $\CA_2$ is given by $\CA_2 =\{1, \sigma_1,\sigma_2,
\sigma_3, \sigma_1\sigma_2,\sigma_2\sigma_1\}$. 

Define operators $\CP^+$ and $\CP^-$ acting on functions $f(\tb)$ by 
\begin{equation} \label{CP^+}
\CP^\pm f(\tb) =  \frac{1}{6} \left[f(\tb) + f(\tb \sigma_1\sigma_2)+
   f(\tb \sigma_2\sigma_1)
        \pm  f(\tb \sigma_1) \pm  f(\tb \sigma_2)  \pm  f(\tb \sigma_3) \right].
\end{equation}
They are projections from the class of $H$-periodic functions onto the class
of invariant, respectively anti-invariant functions under $\CA_2$. The action of
these operators on elementary exponential functions was studied in \cite{K},
and more recently studied in \cite{Sun, LS} and in \cite{LSX}. For $\phi_\kb(\tb)$, 
we call the functions 
$$
 \TC_\kb(\tb) :=  \CP^+ \phi_\kb(\tb), \qquad \hbox{and} \qquad 
  \TS_\kb(\tb): = \frac{1}{i} \CP^- \phi_\kb(\tb) 
$$
a generalized cosine and a generalized sine, respectively.  For invariant 
functions, we can translate the results over the regular hexagon to those 
over one of its six equilateral triangles. We choose the triangle as 
\begin{align} \label{Delta}
   \Delta := & \{(t_1,t_2,t_3) : t_1 + t_2 + t_3 =0,   0 \le t_1,  t_2, -t_3 \le 1\}\\
           = & \{(t_1,t_2): t_1, t_2 \ge 0, \, t_1+t_2 \le 1\}. \notag
\end{align}
It is known that the generalized cosines $\TC_\kb$ are orthogonal with respect 
to the inner product 
$$
   \langle f, g \rangle_\Delta :=
       \frac{1}{|\Delta|}\int_\Delta f(\tb)\overline{g(\tb)} d\tb
      = 2 \int_\Delta f(t_1,t_2) \overline{g(t_1,t_2)} dt_1 dt_2,
$$
so that we can consider orthogonal expansions in terms of generalized 
cosine functions, 
$$
    f \sim \sum_{\kb \in \Lambda} \wh f_\kb \TC_\kb, \qquad 
        \wh f_\kb =   \langle f, \TC_\kb \rangle_\Delta, 
$$
where $\Lambda: = \{\kb \in \HH: k_1 \ge 0, k_2 \ge 0, k_3 \le 0\}$. It is 
known that $\la f,g\ra_H = \la f, g \ra_\Delta$ if $f \bar g$ is invariant. 
If $f$ is $\CA_2$ invariant and $H$-periodic, then it can be expanded
into the generalized cosine series and it can be approximated from the
space 
$$
   \CTC_n^* = \sspan \{\TC_\kb: \kb \in \Lambda, - k_3 \le n\}.
$$
This is similar to the situation in classical Fourier series,  in which even 
functions can be expanded in cosine series and approximated by 
polynomials of cosine. We state one theorem as an example.

\begin{thm} \label{thm:trig} 
If $f \in C(\Delta)$ is the restriction of a $\CA_2$ invariant function 
in $C(\overline{\Omega})$, then the (C,1) means of its 
generalized cosine series converge uniformly to $f$ on $\Delta$. 
\end{thm}

We take this theorem as an example because the meaning of the $(C,1)$ 
means of the generalized cosine series should be clear and we do not
need to introduce any new definition or notation. The requirement of 
$f$ in the theorem may look redundant, but a moment of reflection shows 
that merely $f \in C(\Delta)$ is not enough. Indeed, in the classical Fourier 
analysis, an even function derived from even extension of a function $f$ 
defined on $[0,\pi]$ (by $f(-x) = f(x)$) satisfies $f(-\pi) = f(\pi)$, so that it is 
automatically a continuous $2\pi$ periodic function if $f$ is continuous. The 
$H$-periodicity, however, imposes a much stronger restriction on the function. 
Indeed, for a function $f$ defined on $\Delta$, we can then extend it to 
$F$ defined on $\Omega$ by $\CA_2$ symmetry. That is, we define 
$$
  F(\tb) = f (\tb \sigma), \qquad \tb \in \Delta \sigma, \quad \sigma \in \CA_2.
$$
It is evident that $\Delta = \cup_{\s \in \CA_2} \Delta \s$. In order that $F$
is a continuous $H$-periodic, we will need the restrictions of $F$ on two 
opposite linear boundaries of $\Omega$ are eqaul. Let $\partial_\Omega \Delta $ 
denote the part of boundary of $\Delta$ that is also a part of boundary of 
$\Omega$. Then 
$\partial_\Omega \Delta : = \{((t_1,t_2, -1):  t_1, t_2 \ge 0, t_1+t_2 =1\}$. 
Upon examining the explicit formula of $\s \in \CA_2$, we see that $F$ being 
continuous and $H$-periodic requires that
$$
   F(\tb) = F(\tb \s_2), \quad F(\tb \s_1) = F(\tb \s_2\s_1), \quad 
    F(\tb \s_3) = F(\tb \s_1\s_2), \quad    \tb \in \partial \Delta.
$$
In terms of $f$ this means
\begin{align} \label{f-periodic}
& f(t_1,t_2, -1) = f(-t_2,-t_1,1), \quad f(t_2, -1,t_1) = f(-t_1, 1, -t_2), \\
& f(-1,t_1,t_2) = f(1,-t_2,-t_1), \qquad  
\hbox{for $t_1+t_2 =1$ and $t_1,t_2 \ge 0$}.  \notag
\end{align}
Hence, only functions that satisfy the restrictions \eqref{f-periodic} can
be extended to $\CA_2$ invariant functions on $\Omega$ that are also 
$H$-periodic and continuous on $\Omega$. As a result, we can replace
the assumption on $f$ in Theorem \ref{thm:trig} by {\it if $f \in C(\Delta)$ 
and $f$ satisfies \eqref{f-periodic}.} It does not seem to be easy to classify
all functions that satisfy \eqref{f-periodic}, which is clearly satisfied if $f$
has $\CA_2$ symmetry.


\begin{thebibliography}{99}

\bibitem{AB}
       V. V. Arestov and E. E. Berdysheva, 
       Tur\'an's problem for positive definite functions with supports in a hexagon,
       {\it  Proc. Steklov Inst. Math. 2001, Approximation Theory. Asymptotical 
        Expansions}, suppl. 1, S20--S29.

\bibitem{BX}
        H. Berens and Yuan Xu,
         Fej\'er means for multivariate Fourier series,
         \textit{Math. Z.} \textbf{221} (1996), 449--465.
               
\bibitem{CS}
        J. H. Conway and N. J. A. Sloane, 
        \textit{Sphere Packings, Lattices and Groups},  3rd ed. 
        Springer, New York, 1999.
 
\bibitem{DL}
        R. A. DeVore and G. G. Lorentz,
        \textit{Constructive Approximation}, 
        Springer-Verlag, New Yourk, 1993.

\bibitem{Dunkl}
        C. F. Dunkl, 
        Orthogonal polynomials on the hexagon,   
         \textit{SIAM J. Appl. Math.}, \textbf{47} (1987), 343--351. 

\bibitem{F}
        B.  Fuglede, 
        Commuting self-adjoint partial differential operators 
        and a group theoretic problem,  
        \textit{J. Functional Anal.} \textbf{16} (1974), 101--121.
                  
\bibitem{H}
        J. R. Higgins,
        \textit{Sampling theory in Fourier and Signal Analysis, Foundations},
       Oxford Science Publications, New York, 1996.

\bibitem{Kam}
        A. I. Kamzolov,
        The best approximation on the classes of functions $W_p^\alpha(S^n)$ by
        polynomials in spherical harmonics,
        \textit{Mat. Zametki}, \textbf{32} (1982), 285--293; English transl in
        \textit{Math Notes}, \textbf{32} (1982), 622--628.
                  
\bibitem{K}
        T. Koornwinder, 
        Orthogonal polynomials in two varaibles which are eigenfunctions 
        of two algebraically independent partial differential operators,
        \textit{Nederl. Acad. Wetensch. Proc. Ser. A77} = \textit{Indag. Math}.
        \textbf{36} (1974), 357--381.      
  
\bibitem{LSX}
         H.  Li, J. Sun and Yuan Xu,
         Discrete Fourier analysis, cubature and interpolation on a hexagon 
         and a triangle,
         \textit{SIAM J. Numer. Anal.}  accepted  for publication.
         
              
\bibitem{PX}
         P. Petrushev and Yuan Xu,
         Localized polynomial frames on the interval with Jacobi weights, 
         \textit{J. Fourier Anal. Appl.} 11 (2005), 557-575.  

\bibitem{P}
         A. N. Podkorytov, 
         Asymptotics of the Dirichlet kernel of Fourier sums with respect to a polygon, 
         \textit{J. Soviet Math.} \textbf{42} (1988), no. 2, 1640--1646.
     
\bibitem{Sun}
       J. Sun, 
       Multivariate Fourier series over a class of non tensor-product 
       partition domains, 
       \textit{J. Comput. Math.} \textbf{21} (2003), 53--62.       
       
\bibitem{LS}
        J. Sun and H. Li,
        Generalized Fourier transform on an arbitrary triangular domain,
        \textit{Adv. Comp. Math.}, \textbf{22} (2005), 223-248. 
             
\bibitem{X05} 
          Yuan Xu,
          Weighted approximation of functions on the unit sphere,
          \textit{Constructive Approx.} \textbf{21} (2005), 1--28 

\bibitem{Z}
        A. Zygmund,
        \textit{Trigonometric series},
        Cambridge Univ. Press, Cambridge, 1959. 
        
\end{thebibliography}
 \end{document}